\documentclass[12pt,a4paper,reqno]{article}
 
\usepackage{amsmath,amssymb,amsthm,amscd,dsfont,psfrag,color}
 
\usepackage[english]{babel}
\usepackage[mathscr]{eucal}
\usepackage[dvips]{graphicx}
 
\font\sc=rsfs10 at 12pt
 
\input{comment.sty}
\includecomment{task}
\includecomment{MM}
\includecomment{JK}
\numberwithin{equation}{section}
 
\renewcommand{\a}{\alpha}
\renewcommand{\b}{\beta}
\newcommand{\g}{\gamma}
\newcommand{\G}{\Gamma}
\renewcommand{\d}{\delta}
\newcommand{\D}{\Delta}

\newcommand{\ve}{\varepsilon}

\renewcommand{\th}{\theta}

\renewcommand{\l}{\lambda}

\newcommand{\m}{\mu}
\newcommand{\n}{\nu}
\newcommand{\x}{\xi}

\renewcommand{\r}{\rho}

\renewcommand{\t}{\tau}
\newcommand{\f}{\phi}

\newcommand{\h}{\chi}

\renewcommand{\o}{\omega}
\renewcommand{\O}{\Omega}

\newcommand{\C}{{\mathbb C}}
\newcommand{\R}{{\mathbb R}}

\newcommand{\N}{{\mathbb N}}

\newcommand{\mbf}[1]{\protect\mbox{\boldmath$#1$\unboldmath}}

\newcommand{\Ab}{{\mathbf A}}

\newcommand{\Ac}{{\mathcal A}}
\newcommand{\Bc}{{\mathcal B}}

\newcommand{\Dc}{{\mathcal D}}

\newcommand{\Fc}{{\mathcal F}}

\newcommand{\Oc}{{\mathcal O}}

\newcommand{\pdf}[2]{\frac{\partial {#1}}{\partial {#2}}}

\DeclareMathOperator{\im}{{\rm Im}\,}
\DeclareMathOperator{\re}{{\rm Re}\,}

\newcommand{\Tr}{\operatorname{Tr\,}}
\newcommand{\tr}{\operatorname{tr\,}}

\newcommand{\supp}{\operatorname{supp\,}}

\newcommand{\Res}{\operatorname{Res\,}}

\newcommand{\opc}[1]{{\rm Op}[{#1}]}

\newcommand{\spec}{\operatorname{spec\,}}

\newcommand{\smac}{\operatorname{spec}_{\operatorname{ac}}}

\newcommand{\sme}{\operatorname{spec}_{\operatorname{ess}}}

\newcommand{\ssa}{\operatorname{sing}\operatorname{supp}_{\rm a}}
 
\newcommand{\wfa}{\operatorname{WF}_{\operatorname{a}}}

\newcommand{\ham}[1]{\mathbb{#1}} 
 
\newcommand{\cs}{C^{\infty}} 
\newcommand{\ccs}{C_{0}^{\infty}} 

\newcommand{\ddist}{\sc\mbox{D}^{\hspace{2.0pt} \prime}\hspace{1.0pt}} 
\newcommand{\dtemp}{\sc\mbox{S}^{\hspace{2.0pt} \prime}\hspace{1.0pt}} 
\newcommand{\edist}{\sc\mbox{E}^{\hspace{2.0pt} \prime}\hspace{1.0pt}} 
 
\newcommand{\tS}{\sc\mbox{S}\hspace{1.0pt}} 

\newcommand{\Fr}{\sc\mbox{F}\hspace{0.1pt}} 
 
\newtheorem{theorem}{Theorem}[section]
\newtheorem{proposition}[theorem]{Proposition}
\newtheorem{lemma}[theorem]{Lemma}

\theoremstyle{definition}
\newtheorem{definition}[theorem]{Definition}

\theoremstyle{remark}

\newenvironment{myindentpar}[1]%
 {\begin{list}{}%
         {\setlength{\leftmargin}{#1}}%
         \item[]%
 }
 {\end{list}}

\author{J. Kungsman \\ 
             Department of Mathematics \\
             Uppsala University \\
             SE-751 06 Uppsala, Sweden
             \and 
             Michael Melgaard \\
             Department of Mathematics \\ 
             University of Sussex \\ 
             Brighton BN1 9QH, Great Britain}

\title{Existence of Dirac resonances in the semi-classical limit}

\date{January 14, 2014}

\begin{document}
\maketitle 
\begin{abstract}
We study the existence of quantum resonances of the three-dimensional semiclassical Dirac operator perturbed by smooth, bounded and 
real-valued scalar potentials $V$ decaying like $\langle x \rangle ^{-\d }$ at infinity for some $\d >0$. By studying analytic singularities of a certain 
distribution related to $V$ and by combining two trace formulas, we prove that the perturbed Dirac operators possess resonances near 
$\sup V + 1$ and $\inf V -1$.  We also provide a lower bound for the number of resonances near these points expressed in terms of the semiclassical parameter.     
\end{abstract}
\section{Introduction}
\label{jkmm13aintroduction}
Several approaches have been pursued in the theory of resonances for (nonrelativistic) Schr\"{o}dinger operators, in particular
analytic dilation \cite{aguilar_combes,balslev_combes}, analytic distortion \cite{hunziker} and, in the semiclassical approximation, the one developed by 
Helffer and Sj\"{o}strand \cite{helsjo86}.  When one can simultaneously apply them to an operator, it turns out that the different definitions 
give the same resonances, as demonstrated by Martinez and Helffer \cite{helmar87}.  We refer to  Harrell \cite{harrell07} and Hislop \cite{hislop12} for 
recent surveys. 

Of particular interest for this paper we mention that, for semiclassical Schr\"{o}dinger operators $-\hbar ^2 \D + V(x)$ it was shown by 
Sj\"{o}strand  \cite{sjostrand97_2}, using his local trace formula \cite{sjostrand97},  that analytic singularities of certain distributions defined in terms of 
$V$ produce many resonances near any point of the analytic singular support of the afore-mentioned distributions. These results are also discussed in 
\cite{sjostrand01a,sjostrand01b}. Later,  Nedelec \cite{nedelec01} carried over these results to Schr\"odinger operators with 
matrix-valued potentials by establishing a local trace formula similar to Sj\"{o}strand's.

Resonances for three-dimensional Dirac operators were first studied rigorously by Weder \cite{weder73} and Seba \cite{seba} who, in the spirit of 
Aguilar-Balslev-Combes-Simon theory, applied the method of complex dilation to the Dirac operator. Utilizing the above-mentioned approach by 
Helffer and Sj\"{o}strand, Parisse \cite{parisse91,parisse92} studied resonances in the semiclassical limit, proving the existence of shape resonances, 
located exponentially near the real axis, and establishing the asymptotic behavior of the imaginary part of the first resonance in the case when the potential 
well is localized and non-degenerate.  Amour, Brummelhuis and Nourrigat \cite{amour_brummelhuis_nourrigat_01} proved the existence of resonances 
in the non-relativistic limit and for potentials that behave like a positive power of $|x|$ at infinity.  Khochman \cite{khochman07} studied Dirac operators with 
smooth matrix-valued potentials having an analytic extension in a sector of $\C^{3}$ around $\R^{3}$ outside a compact set and power decay in this sector. 
Within the semiclassical regime and using the complex distortion approach to resonances, he gave an upper bound for the number of resonances in 
complex domains of a certain type and he also established a representation formula for the derivative of the spectral shift function for Dirac operators, 
with potentials of more than cubic decay, in terms of their resonances. In particular, he proved a local trace formula for the perturbed Dirac operator, 
analogous to Sj\"{o}strand's formula valid within the nonrelativistic setting.

We are, however, not aware of any work that deal with the existence of Dirac resonances in the semiclassical limit for more general decaying 
potentials in the spirit of Sj\"{o}strand \cite{sjostrand97_2} and Nedelec \cite{nedelec01}. In the present paper we show that a Dirac operator perturbed 
by a non-zero electric  (scalar) potential $V = vI_4$ with $v \in \cs (\R ^3)$ decaying as $C\langle x \rangle ^{-\d }$ for some constant $C$ and $\d >0$ 
possesses resonances near $\sup v(x) + 1$ and $\inf v(x)-1$ and, furthermore, we establish a lower bound on the number of resonances. 

The main outline of the proof is similar to those in \cite{sjostrand97_2} and \cite{nedelec01} which are based on combining the local trace formula for resonances in the spirit of Sj\"ostrand (see \cite{sjostrand97_2}) and a trace formula of Robert (see \cite{robert_trace}). These trace formulas have been adapted to the Dirac operator in \cite{khochman07} and \cite{bruneau_robert99}, respectively. For a pair of Dirac operators $\ham{D}_0 + v_j(x)I_4 = \ham{D}_j= \mbf{d}_j(x, \hbar D)$, $j=1,2$, it turns out that points in the analytic singular support of a distribution $\o$ 
(in symbols, $\operatorname{sing} \operatorname{supp}_{\rm a} (\o )$) given by 
\begin{align*}
	&\langle \o , \f \rangle _{\Dc ', \Dc } \\
	&= \int \limits _{\R ^6} \Big (\f (\l _{+,2} (x, \x )) - \f (\l _{+,1}(x, \x )) + \f (\l _{-,2} (x, \x )) - \f (\l _	{-,1}(x,\x )) \Big ) \, dx \, d\x ,
\end{align*}   
where the $\l _{\pm , j}$ are the eigenvalues of the symbols $\mbf{d}_j$, generate resonances in their vicinity (see Section \ref{jkmm13aresonances_near_analytic_sing}). To better understand what points belong to $\operatorname{sing} \operatorname{supp}_{\rm a} (\o )$ we prefer to describe $\o $ directly in terms of the potentials $v_j$ and in Section \ref{jkmm13aresonances_near_analytic_sing} it is shown that there is a function $\f$ such that $\o = \pm \f \ast \mu $ where 
$$
	\mu (E) = \frac{d}{dE} \Big ( \int \limits _{v_1(x)\ge E}\,dx - \int \limits _{v_2(x)\ge E}\,dx \Big )
$$
for $E>0$. The difficulty lies in relating points of $\operatorname{sing} \operatorname{supp}_{\rm a} (\mu )$ to those of $\operatorname{sing} \operatorname{supp}_{\rm a} (\o )$ and it turns out that by the properties of $\f $ it is convenient to decompose $\o $ into two parts and exploit the theory of analytic pseudodifferential operators.   
\section{Notation and assumptions}
\label{jkmm13anotation}
The free, or unperturbed, semiclassical Dirac operator is the self-adjoint Friedrichs extension of the symmetric operator 
$$
\ham{D}_0 = -i\hbar \sum _{j=1}^3 \a _j \partial _j + \b , \qquad \partial _j := \pdf{}{x_j},
$$
on $\ccs (\R ^3;\C ^4)$.  Here the $\a _j$ are symmetric $4\times 4$ matrices satisfying the usual anti-commutation relations
$$
\a _j \a _k + \a _k \a _j = 2\d _{jk}I_4, \quad j,k = 1,2,3,4, 
$$ 
(with $\a _4 = \b$) where $I_4$ designates the identity matrix.  The extension, which we also denote by $\ham{D}_0$, acts on 
$L^2(\R^3 ; \C ^4)$ and it has domain $H^1(\R ^3 ; \C ^4)$. It is well-known (see, e.g., \cite{thaller}) that the spectrum of 
$\ham{D}_0$ is purely absolutely continuous and equals 
$$
\spec (\ham{D}_0) = \smac(\ham{D}_0) = (-\infty , -1] \cup [1, \infty ),
$$
We consider a pair $\ham{D}_j = \ham{D}_0 + V_j$, $j=1,2$, of perturbations of $\ham{D}_0$ by scalar potentials 
$V_j(x) = v_j(x)I_4$, where $v_j:\R^3 \to \R $ satisfies: 
\begin{myindentpar}{1cm}
\textbf{Assumption} ($\Ab _\d $):  $v_j\in C_{\rm b}^\infty (\R ^3)$  (space of bounded, infinitely continuously differentiable functions) and 
it has an analytic extension into a sector
$$
C_{\ve ,R_0}:= \{ z \in \C ^3 : |\im z| \le \ve |\re z| , |\re z| > R_0 \}
$$
 for some $\ve \in (0,1)$ and $R_0\ge 0$. 
For any $z\in C_{\ve ,R_0}$ we assume  
$$
|v_j(z)| = \Oc (\langle z \rangle ^{-\d }) \quad \text{for some }\d >0 
$$
and
$$
|v_2(z) - v_1(z)| = \Oc (\langle z \rangle ^{-\d}) \quad \text{for some }\d >3,  
$$
where $\langle z \rangle := (1 + |z|^2)^{1/2}$. 
\end{myindentpar}
By introducing the semiclassical Fourier transform 
$$
(\Fr_\hbar u) (\xi ) = \frac{1}{(2\pi \hbar )^{3/2}} \int \limits _{\R ^3} u(x) e^{-ix\cdot \xi /\hbar } \, dx
$$
we can express $\ham{D}_{j}$ as $\hbar $-pseudodifferential operators $\ham{D}_j = \Fr_\hbar  ^{-1}\mbf{d}_j\Fr _\hbar $, where
the so-called (principal) symbols $\mbf{d}_j (x, \x ) = \sum _{k=1}^3 \a _k \x _k + \b + v_j(x)$ has two-fold 
degenerated eigenvalues
\begin{align}
\l _{\pm , j} (x, \xi )= v_j(x) \pm \langle \x \rangle . 
\label{jkmm13aeigenvalues_symbol}
\end{align}
We see that $\l _{+,j} \ge \inf v_j + 1$ and $\l_{-,j}\le \sup v_j -1$ and, as in \cite{bruneau_robert99}, we define
\begin{align*}
l_{+,j} &= \max (1, \sup v_j - 1), \\
l_{-,j} &= \min (-1, \inf v_j + 1).
\end{align*}

\subsection*{Classical analytic symbols and analytic wavefront set}
We next recall some definitions and properties of analytic wavefront sets (see e.g. \cite[Chapter 8, Section 4]{hšr_1} and \cite[Sections 6,7]{sjšstrand82}) and of so-called 
classical analytic symbols, especially when they are also elliptic (see,  for instance,  \cite[Chapter 5]{treves}). Throughout 
this section we suppress the $\hbar $-dependence (i.e. $\hbar =1$). \\
\\
We say that $u \in \dtemp(\R ^n)$ is of microlocal exponential decay at $(x_0, \x _0) \in \R ^{2n}$ if its Fourier-Bros-Iagolnitzer transform  (in short, FBI-transform) 
$$
(T_\l u) (x,\xi ) = 2^{-n/2}\Big (\frac{\l }{\pi } \Big )^{3n/4} \int e^{i\l (x-y)\cdot \xi  - \l (x-y)^2 /2} \h (y) f(y) \, dy,
$$
where $\h \in \ccs (\R ^3)$ equals 1 near $x_0$, is $\Oc (e^{-C\l })$ near $(x_0, \x _0)$ for some constant $C>0$, uniformly as $\l \to \infty $ (see, e.g., Sj\"{o}strand \cite[Section 6]{sjšstrand82}).

\begin{definition}
The analytic wavefront set of $u\in \dtemp(\R ^n)$, denoted $\wfa (u)$, is the complement in $\R ^n \times (\R ^n \setminus \{ 0 \})$ of the 
set of points, where $u$ is of microlocal exponential decay. 
\end{definition}

It is well-known that $\wfa (u)$ is a closed conic subset of $\R ^n \times (\R ^n \setminus \{ 0 \})$ and that the image of the projection 
onto the first coordinate equals $\operatorname{sing}\operatorname{supp}_{\rm a} (u)$, i.e.\ the smallest closed set outside of which $u$ is real analytic 
(see e.g. Sj\"{o}strand \cite[Section 6]{sjšstrand82}).  
\\
\\
Next we introduce a certain Gevrey class of symbols, namely the analytic one (see, e.g., Treves \cite[Chapter 5]{treves}).    

\begin{definition}
A function $a\in C^\infty (\R ^n \times (\R ^n\setminus \{0\}))$ is said to belong to the space $S_{\rm a}^m(\R ^n)$ of classical analytic symbols if for any $K \Subset \R^n $
$$
|\partial  _x^\alpha \partial _\xi ^\beta a(x,\xi )| \le C^{|\a | + |\b | + 1}\a ! \b ! (1 + |\xi |)^{m - |\b |}
$$
for $x\in K$ and $|\x |\ge B $,  where $B$ and $C$ are positive constants.  
\end{definition}
If $a\in S_{\rm a}^m (\R^n )$, $x_0 \in \R ^n $ and there are constants $C_0,C_1>0$ and a neighborhood $U$ of $x_0$ such that
$$
|a(x,\xi )| \ge C_0 \langle \xi \rangle ^{m} \quad \text{for } x\in U \text{ and }|\xi | \ge C_1
$$
we say that $a$ is elliptic at $x_0$. We say that $a$ (and the corresponding operator $A=a(x,D)$) is elliptic if $a$ is elliptic at every 
$x\in \R^{n}$.  
\begin{proposition}
If $A=a(x,D)$ where $a\in S_{\rm a}^m(\R ^n )$ is elliptic then, for any $u\in \edist (\R ^n)$, we have 
$$
\wfa(Au)= \wfa(u),
$$
where $\wfa(u)$ is the analytic wave front set of $u$. 
\label{jkmm13aelliptic_wavefront}
\end{proposition}
We are only going to use this result in the case $n=1$.

\section{Resonances}
As in Kungsman and Melgaard \cite{jkmm12a} we use the method of complex distortion to define resonances. The method of analytic distortion goes back  to 
Aguilar, Balslev, Combes and Simon \cite{aguilar_combes,balslev_combes,simon73}. In this work we follow the approach by Hunziker \cite{hunziker} 
as implemented by Khochman  \cite{khochman07}. We also state the two trace formulas that we later combine to prove our main results.

\subsection*{Definition of resonances}
Let $R_0\ge 0$ be as in Assumption ($\Ab _\d $) and $g: \R ^3 \to \R ^3$ be a smooth function such that $g(x) = 0$ for $|x|\le R_0$ and $g(x)=x$ outside a compact 
set containing $B(0,R_0)$ which also satisfies $\sup _{x\in \R ^3} \| \nabla g(x) \| \le \sqrt{2}$. Next introduce $\f _\th (x) = x + \th g(x)$ and let $J_\th $ 
denote the Jacobian determinant of $\f _\th $. For $\th \in \R$ we define a one-parameter family of distortions on $\tS (\R ^3;\C ^4)$ by 
$$
(U_\th f)(x) = J_{\th }^{1/2}(x) f(\f _{\th }(x)).
$$
When $|\th | < 2^{-1/2}$ it extends to a unitary operator on $L^2(\R ^3; \C ^4)$ (see \cite{khochman07} for a proof). To define $U_\th $ also for complex 
values of $\th $ one first introduces the linear space $\Ac $ of entire functions $f=(f_1,\ldots ,f_4)$ such that
\begin{align*}
\lim _{\substack{|z|\to \infty  \\ z\in C_{\ve , R_0}}} |z|^k |f_j(z)| = 0, \quad \text{for }1\le j \le 4, \quad k\in \N ,
\end{align*}
and let the dense subspace $\Bc $ of so called analytic vectors consist of the functions in $L^2(\R ^3; \C ^4)$ that have extensions in $\Ac $. With 
$$
D_{\ve } :=\big  \{z \in \C : |z|\le \frac{\ve }{\sqrt{1 + \ve ^2}} \big \}
$$
it then holds that for any $f\in \Bc $ the map $\th \mapsto U_\th f$ is analytic for $\th \in D_\ve $ and $U_\th \Bc $ is dense in $L^2(\R ^3; \C ^4)$ for any 
$\th \in D_\ve $. One can then show that for $\th \in D_{\ve }$ the operator 
$$
U \ham{D}_j U^{-1} = U \ham {D}_0U^{-1} + UV_jU^{-1}=\ham{D}_{0, \th } + V_j\circ \f _\th 
$$
with domain $H^1(\R ^3; \C ^4)$ is an analytic family of type $A$ in the sense of Kato (see \cite[Chapter 7, Section 2]{kato} for the definition of type-A analyticity). It is shown in 
\cite{khochman07} that 
$$
\sme(\ham{D}_\th ) = \spec (\ham{D}_{0, \th }) = \G _\th = \Big \{ z = \pm \Big(\frac{\l }{ (1 + \th )^2} + 1 \Big )^{1/2}, \, \l \in [0,\infty ) \Big \}.
$$ 

The following version of the Aguilar-Balslev-Combes-Simon theorem for the perturbed Dirac operator was established by Khochman \cite{khochman07}. 

\begin{proposition}
For $\th  _0 \in D_{\ve }^+ = D_\ve \cap \{\im z \ge 0\}$ we have 
\begin{itemize}
\item[(i)] For $f,g\in \Bc $, the function 
$$
F_{f,g}(z)=\langle f, (\ham{D} - z)^{-1}g \rangle 
$$
has a meromorphic extension from 
$$
\Sigma = \{ \im z \ge 0 , \, \re z > -1 \} \cup \{ \im z\le 0, \, \re z < 1 \}\setminus \spec (\ham{D}) 
$$
across $\spec (\ham{D})$ and into 
$$
S_{\th _0} = \Big \{ \bigcup _{\th \in D_\ve ^+} \G _\th ; \, \arg (1+\th ) < \arg (1+\th _0), \, \frac{1}{|1 + \th |}< \frac{1}{|1 + \th _0|} \Big \}.
$$  
\item[(ii)] The poles of the continuation of $F_{f,g}$ into $S_{\th _0}$ are the eigenvalues of $\ham{D}_{\th _0}$. 
\item[(iii)] These poles are independent of the family $U_{\th _0}$.
\item[(iv)] The operator $\ham{D}_{\th _0}$ has no discrete spectrum in $\Sigma $.  
\end{itemize}
\label{jkmm13aknochmanprop}
\end{proposition} 

Proposition~\ref{jkmm13aknochmanprop} justifies the following definition. 

\begin{definition}
The resonances of $\ham{D}$ in $S_{\th _0}\cup \R $, denoted $\Res (\ham{D})$, are the eigenvalues of $\ham{D}_{\th _0}$. If $z_0$ is a resonance we take 
its multiplicity to be the rank of the projection
$$
\frac{1}{2\pi i} \int \limits _{\g _{z_0}} (\ham{D}_{\th _0} - z)^{-1} \,dz,
$$
where $\g _{z_0}$ is a sufficiently small positively oriented circle about $z_0$.
\end{definition}
\subsection{Trace formulas}
Herein we recall two trace formulas that will be used in conjunction to give Theorem \ref{jkmm13alower_bound_thm}. This is analogous to 
Sj\"{o}strand \cite{sjostrand97_2} and Nedelec \cite{nedelec01} who studied the case for Schr\"{o}dinger type operators. 

In Khochman \cite{khochman07} the following local trace formula in the spirit of Sj\"{o}strand \cite{sjostrand97} is proved. See also \cite{bruneau_petkov03} and \cite{nedelec01} for 
similar results. To state it we make the following assumption on $\O \subset \C$: 
\newline
\newline
\noindent
\textbf{Assumption} $(\textbf{A}_{\O}^{\pm}):$ $\O $ is an open, simply connected and relatively compact subset of $|\re z |>1$ such that $\O \cap \C _{\pm } \ne \emptyset $ and there exists $\th _0\in D_\ve ^+ $ such that $\overline{\O } \cap \G _{\th _0} = \emptyset $.   
\newline

\noindent
Then one has: 

\begin{theorem}
Suppose $\O \subset \C $ satisfies Assumption $(\Ab _\O ^\pm)$ and suppose, in addition, that $\O \cap \R = I$ is an interval. Let $f$ be a holomorphic function on $\overline{\O }$ 
and $\h \in \ccs (\R )$ (independent of $\hbar $) be equal to 1 near $\overline{I}$. Then, for $v_j$ satisfying Assumption ($\Ab _\d $), we have 
\begin{multline*}
\Tr [(\h f)(\ham{D}_2 )] - \Tr [(\h f)(\ham{D}_1 )] \\ 
= \sum _{z_j \in \Res (\ham{D}_2)\cap \O } f(z_j) - \sum _{z_j \in \Res (\ham{D}_1)\cap \O } f(z_j) + \Oc (\hbar ^{-3}).
\end{multline*}
\label{jkmm13alocal_trace_formula}
\end{theorem}
In \cite[p 21-22]{bruneau_robert99} on the other hand we find the following trace formula by Bruneau and Robert:

\begin{multline}
\Tr [(\h f)(\ham{D}_2 )] - \Tr [(\h f)(\ham{D}_1)] \\
= C\hbar ^{-3} \int \limits_{\R ^6} \Big (\tr [(\h f)(\mbf{d}_2 )] - \tr [(\h f)(\mbf{d}_1 )] \Big ) \, dx\, d\x  + \Oc (\hbar ^{-2}), \\
\label{jkmm13arobert_trace_formula}
\end{multline}
where ``$\tr$'' denotes the matrix trace. 
Here we may write 
\begin{align}
\int \limits _{\R ^6} &\Big (\tr [(\h f)(\mbf{d}_2 )] - \tr [(\h f)(\mbf{d}_1 )] \Big ) \, dx\, d\x  \nonumber \\  
&= 2 \int \limits _{\R ^6} \Big ((\h f) (\l _{+,2}(x, \x )) - (\h f) (\l_{+,1}(x, \x) ) \nonumber \\
&\phantom{ooooooooooooooooooo}+ (\h f) (\l _{-,2}(x, \x )) - (\h f) (\l_{-,1}(x,\x)) \Big )\, dx \, d\x \nonumber \\
&=2\int (\h f)(E) d\r (E),
\label{jkmm13arobert_trace_simplified}
\end{align}
with 
\begin{multline}
\r (E) = \int \limits_{\l _{+,2} (x, \x ) \le E} \, dx \, d\x - \int \limits _{\l _{+,1} (x, \x ) \le E}  \, dx \, d\x \\ - \Big ( \int \limits _{\l _{-,2} (x, \x ) \ge E} \, dx \, d\x - \int \limits _{\l _{-,1} (x, \x ) \ge E} \, dx \, d\x \Big ).  
\label{jkmm13adef_rho}
\end{multline}
\section{Resonances near analytic singularities}
\label{jkmm13aresonances_near_analytic_sing}
In this section we state and prove our main results. 

We introduce $\n _{\pm , j} \in \ddist(\R )$, $j=1,2$, given by 
\begin{align*}
\n _{+ , j } (E) &= -\int \limits _{v_j(x) \ge E} \, dx, \quad &\text{ for }E>0,   \\
\n _{- , j } (E) &= \int \limits _{v_j(x) \le E} \, dx, \quad &\text{ for }E<0,
\end{align*}
with supports equal to $[0,\sup v_j]$ and $[\inf v_j ,0]$, respectively, and we note that $\n_{+,j}$, respectively, 
$\nu_{-,j}$, is a decreasing function, respectively, increasing function. Define, in the sense of distributions, 
$$
\m _{\pm , j} = d\n _{\pm ,j}/dE.
$$ 
Then, since $\n_{+,j}$ are decreasing, $\m_{\pm,j}$ are positive measures (of locally finite mass) on $\R_{\pm} = \pm (0, +\infty )$ and 
\begin{eqnarray*}
\supp \mu_{+,j} & \subset &  \supp \nu_{+,j} = [ 0, \sup v_{j}], \\ 
\supp \mu_{-,j} & \subset &   \supp \n_{-,j} = [ \inf v_{j} , 0 ]. 
\end{eqnarray*}
Moreover, 
\begin{equation*}
\operatorname{sing}\operatorname{supp}_{\rm a} (\mu_{\pm, j}) = \operatorname{sing}\operatorname{supp}_{\rm a} (\nu_{\pm,j}). 
\end{equation*}
We define the distribution $\m \in \ddist (\R )$ by 
\begin{align}
\langle \m , \f \rangle = \int \limits _{\R ^3}\big (\f (v_2(x)) - \f (v_1(x)) \big ) \, dx.
\label{jkmm13amu_def}
\end{align}
Clearly $\supp (\m ) \subset [\min _{j=1,2} \inf _{x\in \R ^3} v_j(x), \max _{j=1,2}\sup _{x\in \R ^3}v_j(x) ]$ 
and since 
$$
|\langle \m , \f \rangle | \le \sup _{t \in \R } |\f '(t)| \int \limits _{\R ^3} |v_2(x) - v_1(x)| \le C \sup _{t \in \R } |\f '(t)| , 
$$
where the last step uses \eqref{jkmm13aeigenvalues_symbol} we see that $\m $ is a distribution of order $\le 1$ (see \cite[Chapter 2, Section 1]{hšr_1}). 
Finally we define $\o \in \ddist (\R)$ by
$$
\langle \o , \f \rangle = \int \limits _{\R ^6} \Big (\f (\l _{+,2} (x, \x )) - \f (\l _{+,1}(x, \x )) + \f (\l _{-,2} (x, \x )) - \f (\l _{-,1}(x,\x )) \Big ) \, dx \, d\x ,
$$
for $\f \in \ccs (\R)$. From \eqref{jkmm13aeigenvalues_symbol} it is clear that the integral with respect to $\x $ can be estimated from above by
$$
2 \operatorname{vol}B(0,R)  \sup _{t\in \R } |v_2(x) - v_1(x)|
$$
for some sufficiently large $R>0$, depending on the support of $\f $. It follows from \eqref{jkmm13aeigenvalues_symbol} that also $\o $ is a distribution of order $\le 1$.  
We remark that for $\f \in \ccs (\R _+)$ we have
$$\langle \m _{+,j} , \f \rangle = -\int \f ' (E) \n _{+, j} (E) \, dE = \int \f (v_j(x))\, dx, $$
so that 
\begin{align}
\m |_{\R _+} =  \m _{+,2} - \m _{+,1}, 
\label{jkmm13amu_restricted}
\end{align}
and similarly for $\m |_{\R _-}$.
\subsection{Main results}
Here we present and discuss the main results, i.e. how certain analytic singularities of $\m $ defined in \eqref{jkmm13amu_def} generate resonances of the Dirac operator. 
We phrase our main theorem for $E_0>0$:

\begin{theorem}
Suppose $\ham{D}_j = \ham{D}_0 + V_j$ where $V_j$, $j=1,2$, satisfies Assumption $(\Ab _\d )$. Let $0<E_0$ be a maximal or a minimal boundary point of $\supp (\m )$. 
Then, for any complex neighborhood $U$ of $E_0+1$ and $E_0-1$, respectively, there exists a constant $C = C (U)$ such that 
$$
\sum _{j=1}^2\# (\Res (\ham{D}_j)\cap U) \ge C\hbar ^{-3},
$$
provided $\hbar $ is small enough. 
\label{jkmm13alower_bound_thm}    
\end{theorem}
Given $v_2$ it is possible to construct $v_1$ so that Assumption ($\Ab _\d $) holds but $\ham{D}_1$ has no resonances near $E_0+1$. We may then invoke the previous 
result to obtain the following result for a single Dirac operator.

\begin{theorem}
Let $0<E_0= \sup v_2(x)$ where $V_2=v_2I_4$ satisfies Assumption $(\Ab _\d )$. Then, for any complex neighborhood $U$ of $E_0+1$ there exists a constant $C = C (U)$ such that 
$$\# (\Res (\ham{D}_2)\cap U) \ge C\hbar ^{-3},$$ 
provided $\hbar $ is small enough. 
\label{jkmm13alower_bound_cor}    
\end{theorem} 

In case $v_2(x)\le 0$ for all $x\in \R$ the above theorems cannot be applied but unless $v_2\equiv 0$ we may then consider $E_0=\min \supp (\m )$ and 
$E_0=\inf v _{2}(x)$, respectively, and resonances near $E_0-1$.

\subsection{Proofs of main results}
The proof of the existence of resonances relies on the fact that we can find points in $\operatorname{WF}_{\rm a}(\o )$. Since the point $(E_0,1)$ in 
Theorem~\ref{jkmm13alower_bound_thm} belongs to $\wfa (\m )$ we begin this section by noting that the distributions $\o $ and $\m $ can be related to each other via convolution. Since the convolution kernel is singular one unit off the diagonal this leads us to decompose $\o $ into two terms corresponding to these singularities. Each term has a symbol which can be represented by a (modified) Bessel function and since the latter happens to be analytic, it enables 
us to apply the theory of analytic pseudodifferential operators mentioned in Section~\ref{jkmm13anotation}. Finally, to prove 
Theorem~\ref{jkmm13alower_bound_cor} we follow the arguments of Sj\"{o}strand \cite{sjostrand97_2} to construct a ``small'' potential $v_1$, given 
$v_2$, so that Assumption ($\Ab _\d $) is fulfilled.  

\begin{lemma}
The decomposition 
$$
\o = \pm A_+ \t _1 \m  \pm  A_- \t _{-1}\m
$$
holds true; the sign corresponds to whether $E>\max _{j=1,2}(l_{+,j})$ or $E<\min _{j=1,2}(l_{-,j})$, respectively, and the $A_{\pm}$ are 
pseudodifferential operators associated with the symbols $\sqrt{2\pi } \Fr [\widetilde{\f }_\pm ](\xi )$, where
\begin{align*}
\widetilde{\f }_+ (x) &= \f _+ (x+1) = (x+1)(x+2)^{1/2}x_+^{1/2} \\
\widetilde{\f }_- (x) &= \f _- (x-1) = (x-1)(2-x)^{1/2}(-x)_+^{1/2},
\end{align*}
and $(\t _{\pm 1}\m )(E) = \m (E \mp 1)$.
\label{jkmm13aomega_decomposition}
\end{lemma} 

\begin{proof}
It is easily verified that 
\begin{align}
\o \, dE = d\r , 
\label{jkmm13aomega_dE_equal_drho}
\end{align}
where $\r $ is defined in \eqref{jkmm13adef_rho}.  By changing to polar coordinates in the $\x $-variable and using $x_+ = \max (x,0)$ for the positive 
part of a real number we may write
\begin{align}
\r (E) &= \pm \frac{8 \pi }{3} \int \limits_{\R ^3}\Big (\big ( (E - v_2(x))^2 - 1 \big )_+^{3/2} - ((E - v_1(x))^2 - 1)_+^{3/2}\Big )\, dx \nonumber  \\
&=\pm \frac{8 \pi }{3} \int  \big ( (E - t)^2 - 1 \big )_+^{3/2} d\n (t)
\label{jkmm13arho_as_convolution}
\end{align}
for $E > \max _{j=1,2}(l_{+,j})$ and $E<\min _{j=1,2}(l_{-,j})$, respectively, where 
$$
\n (t) = - \int \limits_{v_2(x) \ge t} \, dx + \int \limits_{v_1(x)\ge t }\, dx. 
$$
Thus, for any $\f \in \ccs (\R )$, 
$$
\int \f (t) d\n (t) = \int \limits_{\R ^3}\big ( \f (v_2(x)) - \f (v_1(x)) \big ) \, dx = \int \f (E) \m (E) \, dE, 
$$
so $d\n = \m \, dE$. Therefore, we can write $\r $ in \eqref{jkmm13arho_as_convolution} as the convolution
$$
\r = \pm \frac{8 \pi }{3} \big ( (\cdot )^2 - 1 \big )_+^{3/2}\ast \m . 
$$
Consequently, we obtain from \eqref{jkmm13aomega_dE_equal_drho} that 
\begin{align}
\o = \pm  \f  \ast \m 
\label{jkmm13aconvolution_phi_mu}
\end{align}
with 
$$
\f (x) = 8\pi \big ( x^2 - 1 \big )_+^{1/2}x. 
$$
Introduce
\begin{align*}
\f _+(x) &= 8\pi x(x+1)^{1/2} (x-1)^{1/2}_+, \\
\f _-(x) &= 8\pi x(1-x)^{1/2} (-(x+1))^{1/2}_+
\end{align*}
so that $\f  = \f _1 +\f _2$. The convolution \eqref{jkmm13aconvolution_phi_mu} can then be written as
$$
\o (E) = \pm \int \phi _+ (E-t+1)\mu (t-1)\, dt \pm  \int \phi _- (E-t-1)\mu (t+1)\, dt. 
$$
We can thus write 
\begin{align}
\o  = \pm \widetilde{\f }_+ \ast (\t _{+1} \m ) \pm  \widetilde{\f }_- \ast (\t _{-1}\m ) = \pm A_+ \t _1 \m  \pm  A_- \t _{-1}\m , 
\label{jkmm13aomega-decomposition}
\end{align}
where $(\t _{\pm 1}\m )(t)= \m (t\mp 1)$ and $A_{\pm}$ are pseudodifferential operators having symbols $\sqrt{2\pi } \Fr [\widetilde{\f }_\pm ](\xi )$.
\end{proof}

We now show that the pseudodifferential operators $A_{\pm }$ in the previous lemma are elliptic classical analytic. 

\begin{lemma}
With $\widetilde{\f }_\pm $ as in Lemma \ref{jkmm13aomega_decomposition} one has that $\Fr [\widetilde{\f }_\pm ] \in S _{\rm a} ^{-5/2}$ and they are elliptic. 
\label{jkmm13afourier_transform_phi_tilde_elliptic}
\end{lemma}

\begin{proof}
It suffices to prove the result for $\widetilde{\f }_+$. For $\xi \ne 0$ we have (see \cite[Chapter 9, Section 6]{abramowitz_stegun_64})
$$
\Fr [\widetilde{\f }_+](\xi ) = C_0 \frac{e^{i\xi }}{\xi ^2} \Big ( K_1(i\xi ) -i \xi K_1'(i\xi ) \Big ), \quad C_0\ne 0,
$$
where $K_1$ is a modified Bessel function of the second kind. Using the recurrence relation 
$$
K_1(z)  - zK_1'(z) = z K_2(z)
$$
(see \cite[Chapter 9, Section 6]{abramowitz_stegun_64}) we may write
$$
\Fr [\widetilde{\f }_+](\xi  ) = -C_0 \frac{e^{i\xi }}{i\xi }K_2(i\xi ). 
$$
Since $z \mapsto K_2(iz )$ is analytic for $\re z \ne 0$ we can use the Cauchy integral formula in the form
\begin{align}
|D^N \Fc [\widetilde{\f }_+](\xi )| &= \frac{N!}{2\pi } \Big | \int \limits _{|z - \xi |=(1 + |\xi | )/2} \frac{\Fr [\widetilde{\f }_+](z)}{(z - \xi )^{N+1}}\, dz \Big | \nonumber  \\
&\le C2^{N+1}N! (1+|\xi |)^{-N-1} \sup _{|z - \xi |=(1 + |\xi | )/2} |\Fr [\widetilde{\f }_+](z)|, 
\label{jkmm13acauchy_integral_estimate}
\end{align}
for $|\xi | \ge C_0$. Since $|z - \xi |=(1 + |\xi | )/2$ implies that $(|\xi | - 1)/2\le |z|$ we can use the fact that (see e.g. \cite{abramowitz_stegun_64})
\begin{align}
K_2(iz) = \sqrt{\frac{\pi }{2iz}}e^{-iz } \big (1 + \Oc (\frac{1}{|z|}) \big )
\label{jkmm13aK_2_asymptotic_expansion}
\end{align}
for $|z|$ large. It follows that 
$$
 \sup _{|z - \xi  |=(1 + |\xi | )/2} |\Fr [\widetilde{\f }_+](\xi )| \le C  \sup _{|z - \xi |=(1 + |\xi | )/2}|z|^{-3/2} \le C(1+|\xi  |)^{-3/2}.
$$
Together with \eqref{jkmm13acauchy_integral_estimate} it follows that 
$$
|D^N \Fr [\widetilde{\f }_+](\xi )| \le C\cdot 2^{N+1}N! (1+|\xi |)^{-5/2-N} \quad \text{for }|\xi | \ge C_0,
$$
which means that $\Fr [\widetilde{\f }_+] \in S_{\rm a}^{-5/2}(\R )$. 

Finally, it follows from \eqref{jkmm13aK_2_asymptotic_expansion} that
$$
|\Fr [\widetilde{\f }_+] (\xi )|\ge C(1 + |\xi |) ^{-5/2}
$$
provided $|\xi |$ is large enough which shows that $A_+=\sqrt{2\pi } \opc {\Fr [\widetilde{\f }_+]}$ is elliptic.
\end{proof}
\begin{proof}[Proof of Theorem \ref{jkmm13alower_bound_thm}]
Let $E_0= \max  \supp (\m )$ be a maximal boundary point of $\supp (\m )$. Then \cite[Corollary 8.4.16]{hšr_1} asserts that 
$(E_0,\pm 1) \in \wfa(\mu )$. Since $\tau_{+1}$ is a unit shift operator we have that $(E_0 +1,\pm 1) \in \wfa(\t _{+1} \m )$ and 
Proposition~\ref{jkmm13aelliptic_wavefront} implies that, in view of \eqref{jkmm13aomega_decomposition} and 
Lemma~\ref{jkmm13afourier_transform_phi_tilde_elliptic}, we have $(E_0 +1,\pm 1) \in \wfa(\o )$. Let us for the sake of notation concentrate on the point 
$E_0+1$. The crux of the proof is finding the relationship between the properties of $\ssa (\o)$ and the resonances of $\ham{D}_{j}$. 
Utilizing the definition of the wavefront set via the FBI transform (see Section~\ref{jkmm13anotation}) there are real sequences 
$\a _j \to E_0+1$, $\l _j \to +\infty $, $\ve _j\to 0$ and $\b _j \to 1$ such that 
\begin{align}
\int e^{i\l _j \b _j (\a _j - E)-\l _j (\a _j -E)^2/2}\h (E) \o (E) \, dE \ge e^{-\ve _j \l _j} 
\label{jkmm13aFBI_analytic_wave_front_sequence}
\end{align} 
for $\h \in \ccs{(\R _{\ge 1})}$ equal to 1 near $E_0+1$.   
We define the function $f_j(E) = e^{i\l _j \b _j (\a _j - E)-\l _j (\a _j -E)^2/2}$. 
Let $a, b>0$ be two sufficiently small constants such that also $a/b$ is small, and put 
\begin{align*}
\O &= (E_0+1-2b, E_0+1+2b) + i(-2a,a], \\
W &= [E_0+1-b,E_0+1+b]+i(-a,a].
\end{align*}
It is easy to see that $|f_j(E)| \le e^{-C_0\l _j}$ for $E\in \O \setminus W$ for large values of $j$ and appropriate $a$ and $b$. By Theorem \ref{jkmm13alocal_trace_formula}
\begin{multline*}
\Tr [(\h f_j)(\ham{D}_2) - (\h f_j)(\ham{D}_1)] =\\
 \sum _{z_k\in \Res (\ham{D}_2)\cap W} f_j(z_k) - \sum _{z_k\in \Res (\ham{D}_1)\cap W} f_j(z_k) + e^{-C_0\l _j} \Oc (\hbar ^{-3}), 
\end{multline*}
uniformly in $k$. By combining the Bruneau-Robert trace formula \eqref{jkmm13arobert_trace_formula}, \eqref{jkmm13arobert_trace_simplified} and \eqref{jkmm13aomega_dE_equal_drho} we obtain
\begin{multline*}
C\hbar ^{-3} \int f_j (E)  \h (E) \o (E) \,dE \\ = \sum _{z_k\in \Res (\ham{D}_2)\cap W} f_j(z_k) - \sum _{z_k\in \Res (\ham{D}_1)\cap W} f_j(z_k) + e^{-C_0\l _j} \Oc (\hbar ^{-3}) + \Oc (\hbar ^{-2})
\end{multline*}
From this and \eqref{jkmm13aFBI_analytic_wave_front_sequence} we obtain
\begin{multline*}
C\hbar ^{-3}e^{-\ve _j \l _j} \le \Big |\sum _{z_k\in \Res (\ham{D}_2)\cap W} f_j(z_k) - \sum _{z_k\in \Res (\ham{D}_1)\cap W} f_j(z_k)\Big | \\ +  \Oc (\hbar ^{-3})e^{-C_0\l _j} + \Oc (\hbar ^{-2})
\end{multline*}
Combined with \eqref{jkmm13arobert_trace_formula}, \eqref{jkmm13aomega_dE_equal_drho} and \eqref{jkmm13aFBI_analytic_wave_front_sequence} this gives
\begin{multline*}
 \Big | \sum _{z_k\in \Res (\ham{D}_2)\cap W} f_j(z_k) - \sum _{z_k\in \Res (\ham{D}_1)\cap W} f_j(z_k) \Big | \\
 \ge C\hbar ^{-3}(e^{-\ve _j \l _j } - \Oc(1)e^{-C_0\l _j})\hbar ^{-3} + \Oc (\hbar ^{-2}) \ge C \hbar ^{-3} 
\end{multline*}
for some $C>0$, where the last inequality follows by fixing a sufficiently large $j$ and then taking $\hbar $ small enough. Since $|f_j|$ is bounded on $W$ the result follows. 
\end{proof}
Theorem \ref{jkmm13alower_bound_cor} now follows from Theorem \ref{jkmm13alower_bound_thm} by constructing the potential $v_1$ so that it produces no resonances. To achieve this we follow the argument outlined in \cite{sjostrand97_2}. By multiplying $v_2$ by a smooth cut-off function we arrange so that it equals 0 in some large ball $B(0,R)$ and is small in its complement. Then we follow this with an appropriate regularization so that complex distortion with $R_0=0$ can be done.  
\begin{proof}[Proof of Theorem \ref{jkmm13alower_bound_cor}]
Let 
$$
K(y) = C_0e^{-y^2/2} \quad \text{where }C_0= \big (\int \limits_{\R ^3}e^{-y^2/2}\, dy \big)^{-1}. 
$$
Put $K_{R} (y) = \l  ^{-3} K (\l ^{-1} y)$ where $\l = \l _R(x) = \langle R^{-1} x \rangle ^{-4}$ where $R\gg 1$ will be specified below. Take $\h \in \ccs (\R )$ which equals 1 for $|x|< 1$ and 0 for $|x|>2$. We define 
$$
v_1(x) = \int \limits_{\R ^3}K_{R } (x-y) \big ( 1 - \h (R^{-1}y) \big ) v_2 (y) \, dy.
$$ 
We see that $v_1$ extends to a holomorphic function in the domain $\{ |\im z| < \ve \langle \re z \rangle , |\re z|\ge 0 \}$  for $\ve  < 1$. Moreover
\begin{align*}
|v_1(x)| \le \int \limits _{\R ^3} K_R(x-y) |v_2(y)|\,dy = C_0\int \limits _{\R ^3} e^{-w^2/2}|v_2(x + \l w)|\,dw.
\end{align*}
Then Assumption ($\Ab _\d $) together with Peetre's inequality in the form
$$
\langle x + \l w\rangle ^{-\d } \le 2^{\d /2} \langle x \rangle ^{-\d } \langle \l w \rangle ^\d \le 2^{\d /2} \langle x \rangle ^{-\d } \langle w \rangle ^\d   
$$ 
implies that
$$
|v_1(x)| \le  C  \langle x \rangle ^{-\d } \int \limits _{\R ^3}e^{-w^2/2}  \langle w \rangle ^\d \,dw  \le C  \langle x \rangle ^{-\d }
$$
for some $\d >0$. Next we find that 
\begin{align*}
&|v_2(x) - v_1(x)| = \big |\int \limits _{\R ^3} K_R(x-y) \Big( v_2(x) - (1 - \h (R^{-1}y))v_2(y)\Big )\, dy \big | \\
&\le \int \limits_{\R ^3} K_R(x-y)|v_2(x) - v_2(y)|\, dy + \int \limits_{|y|<2R}K_R (x-y)\h (R^{-1}y) |v_2(y)| \, dy.
\end{align*}
For the first term we have, by virtue of Assumption ($\Ab _\d $), 
$$
\int \limits_{\R ^3} K_R(x-y)|v_2(x) - v_2(y)|\, dy \le C \int \limits_{\R ^3} K_R(x-y) |x-y|\, dy 
$$
and by making the change of coordinates $y=x+\l w$ we readily see that this can be bounded from above by $C\l $ for some constant $C$, independent of $x$ and $R$. The second term is clearly bounded and for $|x|>4R$ we have $|x-y|\ge |x|/2$ so that it can be bounded from above by 
$$
\lambda ^{-3} e^{-|x|^2/(8\l ^2)}\int \limits _{|y|<2R} |v_2(y)|\, dy \le C\langle x \rangle ^{-\d } 
$$
for any $\d >0$. This shows that 
$$
|v_2(x) - v_1(x)| \le C\langle x \rangle ^{-\d } \qquad \text{for }\d =4,
$$
and thus Assumption ($\Ab _\d $) is fulfilled with $R_0=0$. \\
\\
Finally, we point out that given any $\ve _0>0$ we can construct $v_1$ above so that $|v_1(x)|\le \ve _0$ for all $x\in \R ^3$. Indeed, choose $R>0$ so large that $|v_2(y)|\le \ve _0$ for $|y|>R$. Then   
$$
|v_1(x)| \le \int \limits _{|y|>R} K_R(x-y) |v_2(y)| \, dy \le \ve _0.
$$
We next show that the ``small'' potential $v_1$ constructed above cannot generate any resonances near $E_0+1$. \\
\\
If $z_0=E_0+1+\ve _0$ is a resonance close to $E_0+1$ we can compute $\ham{D}_\th $ explicitly (see \cite{khochman07}) and find $u\in L^2(\R ^3; \C ^4)$ with $\| u \| = 1$ such that 
$$
-\frac{1}{1+\th } i\hbar \sum _{j=1}^3 \a _j \partial _j u + \b u + (V_1\circ \f _\th ) u = (E_0 +1+ \ve _0)u
$$ 
which can be rewritten as 
$$
\| \ham {D}_0u - (E_0+1)u \| = \Oc (|\th |) + \Oc (|\ve _0|) + \Oc (\sup |V_1\circ \f _{\th }|)
$$
Since the norm appearing on the left hand side is independent of the quantities on the right hand side, and since these quantities can be made arbitrarily small, we see that $\ham {D} _0u = (E_0+1)u$ which is a contradiction. \\
\\
To prove the theorem it suffices that we construct $v_1$ as above such that the bound $\sup _{x\in \R ^3}|v_1(x)| < E_0/2$ holds. 
Then we  have $(E_0,1)\in \wfa(\n _{+,2})$ because $E_0 = \sup _{x\in \R ^3}v_2(x)$ is the right end point of $\supp (\n _{+,2})$.  
Since $d/dE$ is an elliptic analytic operator it follows that $(E_0,1)\in \wfa(\m _{+,2})$. Consequently, since $\m _{+,1}(E)=0$ for 
$E\ge E_0/2$  we infer from \eqref{jkmm13amu_restricted} that $(E_0,1)\in \wfa(\m )$ and the proof of  
Theorem~\ref{jkmm13alower_bound_thm} applies.    
\end{proof}

\textbf{Acknowledgement:} The first author is grateful to Johannes Sj\"ostrand for an illuminating discussion and the second author thanks Johannes Sj\"{o}strand for introducing him 
to semiclassical analysis and resonances at University of Gothenburg and Chalmers University of Technology during the spring terms of 2001-2002.

\end{document}